\documentclass[11pt]{article}
\usepackage{amsfonts,epsf,amsmath,mathrsfs}
\usepackage{amssymb}
\usepackage{graphicx}
\usepackage{tikz}
\usepackage[cp1250]{inputenc}
\usepackage{amssymb,amsfonts,amsmath, cite}
\usepackage{enumerate}

\usepackage[english]{babel}
\usepackage[cp1250]{inputenc}

\usepackage{algorithm2e}

\usepackage{longtable}

\usepackage{graphicx}
\usepackage{index}
\usepackage{fancyhdr}
\usepackage{lhelp}
\usepackage{optparams}
\usepackage{psfrag}
\usepackage{setspace}
\usepackage{tikz}
\usepackage{subcaption}
\usepackage{pgffor}

\definecolor{lgray}{gray}{0.75}

\usetikzlibrary{decorations.pathreplacing,automata,calc,positioning}
\usepackage{tikz}

\usetikzlibrary{decorations.pathreplacing,automata,calc,positioning}
\usetikzlibrary{arrows}

\newtheorem{theorem}{\bf Theorem}
\newtheorem{corollary}[theorem]{\bf Corollary}
\newtheorem{lemma}[theorem]{\bf Lemma}
\newtheorem{proposition}[theorem]{\bf Proposition}
\newtheorem{problem}{\bf Problem}
\newtheorem{question}{\bf Question}

\newtheorem{remark}[theorem]{\bf Remark}

\newcommand{\proof}{\noindent{\bf Proof.\ }}
\newcommand{\qed}{\hfill $\Box$ \bigskip}

\newcommand{\cig}{\chi_{ig}}
\newcommand{\cigas}{\chi_{ig}^{As}}
\newcommand{\cigA}{\chi_{ig}^{A}}
\newcommand{\cigB}{\chi_{ig}^{B}}
\newcommand{\cigAB}{\chi_{ig}^{AB}}
\newcommand{\cigBA}{\chi_{ig}^{BA}}

\begin{document}

\title{The independence coloring game on graphs}
\author{
Bo\v stjan Bre\v sar$^{a,b}$
\and
Da\v sa \v Stesl$^{a}$
}

\date{\empty}
\maketitle

\vspace{-5mm}

\begin{center}
$^a$ Faculty of Natural Sciences and Mathematics, University of Maribor, Slovenia\\
\medskip

$^b$ Institute of Mathematics, Physics and Mechanics, Ljubljana, Slovenia\\
\medskip
\end{center}

\begin{abstract}
We propose a new coloring game on a graph, called the independence coloring game, which is played by two players with opposite goals. The result of the game is a proper coloring of vertices of a graph $G$, and Alice's goal is that as few colors as possible are used during the game, while Bob wants to maximize the number of colors.  The game consists of rounds, and in round $i$, where $i=1,2,,\ldots$, the players are taking turns in selecting a previously unselected vertex of $G$ and giving it color $i$ (hence, in each round the selected vertices form an independent set). The game ends when all vertices of $G$ are selected (and thus colored), and the total number of rounds during the game when both players are playing optimally with respect to their goals, is called the independence game chromatic number, $\cig(G)$, of $G$. In fact, four different versions of the independence game chromatic number are considered, which depend on who starts a game and who starts next rounds. We prove that the new invariants lie between the chromatic number of a graph and the maximum degree plus $1$, and characterize the graphs in which each of the four versions of the game invariant equals $2$. We compare the versions of the independence game chromatic number among themselves and with the classical game chromatic number. In addition, we prove that the independence game chromatic number of a tree can be arbitrarily large. 
\end{abstract}

\bigskip


%%%%%%%%%%%%%%%%%%%%%%%%%%%%%%%%%%%%%%%%%%%%%%%%%%%%%%%

\section{Introduction}
A game counterpart of the graph coloring problem was introduced independently by Brams (see Gardner~\cite{ga-81}) and Boedlander\cite{bo-1991} with the original motivation to shed some light on the four-color problem and on the complexity issues related to graph coloring. Yet, the coloring game turned out to be of independent interest, and a number of authors studied this game and several of its variations; see two surveys~\cite{bagr-07,tu-2016} and four papers introducing additional versions of graph games that are related to coloring~\cite{bosek, gr-2012,kir-2012,mpw-2018}. One of the natural related games, the marking game, provided an application to a non-game problem of packing; see Kierstead and Kostochka~\cite{kiko-2009}. Most of the mentioned games are Maker-Breaker-type games, in which the set of colors is fixed in advance, and two players with opposite goals are trying to win the game using the given set of colors. The coloring game proposed in this paper is of different nature, which is similar to the domination game, as introduced in~\cite{brklra-2010}. This well studied type of graph game is also played by two players with opposite goals, yet the result of the game is not the win of one of the players, but a graph invariant, which is given by the total number of moves in a game, in which both players play optimally according to their goals.

The {\em independence coloring game} is a game played on a simple graph $G$ by two players, Alice and Bob.
The game consists of rounds. In each round, the players take turns by selecting a previously uncolored vertex and giving it the color of the round. In the first round vertices are colored by color $1$, in the second round they are colored by color $2$, and so on. In each round, the set of chosen vertices must be an independent set. A round ends when none of the unchosen vertices can be properly colored by the color of the round. That is, the first round ends when the set of chosen vertices is a maximal independent set of $G$, while for $k\ge 2$, the $k$th round ends when the set of chosen vertices in that round is a maximal independent set in the graph $G-C^1\cup \cdots \cup C^{k-1}$, where $C^i$ denotes the set of vertices chosen in round $i$ for each $i\in\{1,2,\ldots ,k-1\}$. The game ends when all vertices have been chosen. Clearly, the total number of rounds coincides with the number of colors that are used in the entire game, and the resulting coloring is a proper coloring. The goal of Alice is to minimize the number of rounds (colors), while Bob wants to maximize that number. 

We consider four types of games depending on who starts each round. If Alice (resp. Bob) starts the game, and each further round is started by the player who did not end the previous round, then we speak about an {\em AB-independence coloring game} (respectively, {\em BA-independence coloring game}). The number of rounds played on a graph $G$ when both players are playing optimally in an AB-independence coloring game (respectively, BA-independence coloring game) is the {\em AB-independence game chromatic number}, $\cigAB(G)$, (respectively, 
{\em BA-independence game chromatic number}, $\cigBA(G)$) of $G$. On the other hand, if Alice (Bob) starts each round in the game, then this is called an {\em A-independence coloring game} (respectively, a {\em B-independence coloring game}), and the resulting number of rounds is the {\em A-independence game chromatic number}, $\cigA(G)$, (respectively, the {\em B-independence game chromatic number}, $\cigB(G)$) of $G$. Finally, we speak about the {\em independence coloring game} on $G$ and the {\em independence game chromatic number}, $\cig(G)$, when it is not important who starts the game and further rounds (that is, the corresponding result does not depend on the type of the independence coloring game).

The independence coloring game is closely related to the {\em independent domination game}, initiated by Philips and Slater\cite{ph-sl-2001} under the name {\em competition-independence game}, and developed by Goddard and Henning~\cite{go-he-2018}. The game is played by two players, Diminisher and Sweller, on a graph $G$, who are taking turns in constructing a maximal independent set of $G$. The goal of Diminisher is to make the final set as small as possible and Sweller's goal is just the opposite. Note that 
the rules of the first round of the independence coloring game coincide with the rules of the competition-independence game. However, it is not clear under which conditions the first round of our game and the competition-independence game coincide. 

The independence game chromatic number shares two general bounds with many other coloring games. Firstly, it is clear that the chromatic number $\chi(G)$ is a lower bound for the independence game chromatic number, $\cig(G)$, for an arbitrary graph $G$. Secondly,  $\cig(G)\le \Delta + 1$, where $\Delta$ is the maximum degree of vertices in $G$. On the other hand, unlike in the standard coloring game, the independence coloring game in trees and planar graphs may require an arbitrarily large number of colors (see Section~\ref{sec:trees}). In the proof of the result about trees, we introduce a version of the game in which Alice may skip moves and also use a certain type of imagination strategy, which are established tools in the study of domination games~\cite{brklra-2010,KWZ-2013,klra-2019}.  

In Section~\ref{sec:not-bas} we fix the notation and present some basic results. We prove the mentioned upper bound for the independence game chromatic number of a graph, and present exact values for all versions of the independence game chromatic numbers in paths and cycles. In Section~\ref{sec:extremal}, we present some families of graphs that achieve extremal values of the independence game chromatic numbers. While it is clear that $\cig(G)=1$ if and only if $G$ is an edgeless graph, we characterize the graphs $G$, whose independence game chromatic  numbers equal $2$ in each of the four types of the games. We also present a large family of cubic graphs with $\cig(G)=4$, and prove that split graphs $G$ satisfy $\cigA(G)=\chi(G)$. (A graph $G$ is a {\em split graph} if the vertex set of $G$ can be partitioned into two parts one inducing an independent set and the other one inducing a complete graph.) Section~\ref{sec:trees} is devoted to one of our main results that for every positive integer $k$ there exists a tree whose independence game chromatic number is greater than $k$ (we give a formal proof of this result for three of the four independence game chromatic numbers).  In Section~\ref{sec:compare} we compare different versions of the independence game chromatic numbers, and prove that the difference between all ordered pairs of the invariants, with only two possible exceptions, can be arbitrarily large (the question remains whether $\cigAB(G)-\cigA(G)$ and $\cigBA(G)-\cigA(G)$ can be positive in some graphs $G$ and how large it can be). Similarly, in Section~\ref{sec:othergames} we compare the independence game chromatic numbers with the game chromatic number. We end the paper with concluding remarks and open problems.

\section{Notation and basic results}
\label{sec:not-bas}
Throughout this paper we consider simple, undirected graphs. 
The {\em neighborhood} of a vertex $v\in V(G)$, denoted by $N_G(v)$, is the set of vertices adjacent to $v$, and the {\em degree}, $\deg_G(v)$ of $v\in V(G)$, is the cardinality of its open neighborhood. The {\em maximum degree} of vertices in $G$ is denoted by $\Delta(G)$. A graph is {\em cubic} if all its vertices have degree $3$. The graph $K_4-e$ is also called the {\em diamond}. The {\em neighborhood} of a set $S\subseteq V(G)$ is $N_G(S)=\cup_{v\in S}{N_G(v)}$. The {\em distance} between two vertices $x$ and $y$ of a connected graph $G$ is the length of a shortest path between them, and is denoted by $d_G(x,y)$. The index $G$ in the above definitions may be omitted if the graph $G$ is understood from the context. For a positive integer $n$, we write $\{1,\ldots,n\}$ shortly as $[n]$.

When Alice and Bob are playing an independence coloring game, they take turns and select (choose) a vertex of a graph $G$, which has not been chosen earlier and give it the color of the round as long as this is possible. 
In the course of the proofs, we will often say that vertices, which are neighbors of vertices that have been chosen in a given round are {\em protected}. This is usually in favour of Bob, since protecting some vertices prevents Alice to play them in the current round.

We start with the following general bound, which holds for all version of the game independence coloring numbers.

\begin{theorem}
For every graph $G$, $\chi (G)\leq \chi_{ig} (G)\leq \Delta (G)+1$.
\label{prop2}
\end{theorem}

\begin{proof}
The first inequality is clear, since every coloring of vertices obtained in the independence coloring game is a proper coloring of $V(G)$ by the condition that in each round an independent set is built. 

Let Alice and Bob play an independence coloring game on $G$, and let $v$ be an arbitrary vertex of $G$. If after $\deg(v)$ rounds, $v$ has not yet been chosen, then all neighbors of $v$ are colored. This is because a round is not over as long as the players can choose a vertex that is not adjacent to the previously chosen vertices of that round. Hence, either $v$ has been chosen in the first $\deg(v)$ rounds, or $v$ is chosen in the round $\deg(v)+1$. Indeed, if $v$ is not chosen in the first $\deg(v)$ rounds, then $v$ is an isolated vertex in $G-C^1\cup \cdots \cup C^{\deg(v)}$, hence it must be chosen in the round $\deg(v)+1$.  Since this holds for every vertex of $G$, all vertices of $G$ will be colored in round $\Delta(G)+1$ (if the game lasts that many rounds) or earlier. Hence, $\chi_{ig} (G)\leq \Delta (G)+1$. \qed
\end{proof}

It is clear that complete graphs attain both the upper and the lower bound in Theorem~\ref{prop2}, since $\chi(K_n)=\chi_{ig}(K_n)=n=\Delta(K_n)+1$. In Section~\ref{sec:extremal}, we will present more families of graphs that attain both bounds, and will also characterize the families of graphs that attain the value $2$ for each version of the independence game chromatic number. 

In the rest of the section, we present exact values for the independence game chromatic numbers of paths and cycles. In the proofs, the vertices of $P_n$ and $C_n$ will be denoted by $v_1,\ldots, v_n$ with adjacencies defined in the natural way.

\begin{proposition}
\label{prp:_pathsA}
If $P_n$ is a path on $n$ vertices, then
\begin{displaymath}
\cigA(P_n)=\cigAB (P_n) = \left\{ \begin{array}{ll}
1, & \textrm{$n=1$,}\\
2, & \textrm{$2\leq n\leq 5$,}\\
3, & \textrm{$n\geq 6$.}
\end{array} \right.
\end{displaymath}
\label{prop_poti}
\end{proposition}
\begin{proof}
The values of $\cigA(P_n)$ and $\cigAB(P_n)$, when $n\in [3]$, are trivial. If $n=4$, then Alice has a winning strategy in the A-independence and the AB-independence coloring game on $P_n$ by choosing $v_2$ or $v_3$ in her first turn. The next move of Bob is then forced (he must select $v_4$, or $v_1$, respectively), and in the next round only two non-adjacent vertices remain, hence the game is finished after two rounds. Thus, $\cigA (P_4)\leq 2$, $\cigAB (P_4)\leq 2$, while $\cigA (P_4)>1$ and $\cigAB(P_4)>1$ is clear, since $P_4$ is not an edgeless graph. For $n=5$, the optimal first move of Alice in both games is to choose the vertex $v_3$. Consequently, the vertices $v_1$ and $v_5$ will also receive color $1$, and so the vertices $v_2$ and $v_4$ are left for round $2$, in which they will receive color $2$. Therefore, $\cigA (P_5)=2$ and $\cigAB (P_5)= 2$. In the case when $n\geq 6$, Alice cannot achieve that only two rounds will be played in an A-independence or an AB-independence coloring game on $P_n$. Indeed, regardless of which vertex she chooses in her first move, Bob will select a vertex at distance $3$ from Alice's first move. Then, the two vertices that lie between the vertices chosen in the first two moves can no longer receive color $1$ and cannot both be colored with $2$, since they are adjacent. Thus $\cigA (P_n)\geq 3$ and $\cigAB (P_n)\geq 3$. The reversed inequalities, $\cigA (P_n)\leq 3$ and $\cigAB (P_n)\leq 3$, follow by Theorem~\ref{prop2}. Hence, $\cigA (P_n)=\cigAB (P_n)=3$ for every $n\geq 6$. \qed
\end{proof}

\begin{proposition}
\label{prp:_pathsB}
If $P_n$ is a path on $n$ vertices, then
\begin{displaymath}
\cigB(P_n)=\cigBA (P_n) = \left\{ \begin{array}{ll}
1, & \textrm{$n=1$,}\\
2, & \textrm{$2\leq n\leq 6$,}\\
3, & \textrm{$n\geq 7$.}
\end{array} \right.
\end{displaymath}
\end{proposition}
\begin{proof}
Note that in the B-independence and the BA-independence coloring game Bob starts the game. The values of $\cigB(P_n)$ and $\cigBA(P_n)$, when $n\in [3]$, are trivial. Further, if $n\in \{4,5\}$, then Alice can ensure that just two colors will be played in a game, by responding with a vertex, which is at distance $2$ from the first vertex chosen by Bob.  This implies that either there are no more moves in round $1$ and $V(P_n)-C^1$ is an independent set, or, there is just one more move in round $1$, which is forced, and again $V(P_n)-C^1$ is an independent set.  Thus, $\cigB (P_n)\leq 2$ and $\cigBA (P_n)\leq 2$, while $\cigB (P_n)>1$ and $\cigB(P_n)>1$ is obvious. Now, consider the case when $n=6$. By symmetry, we may assume that the first move of Bob is one of the vertices in $\{v_1,v_2,v_3\}$. If Bob's first move is selecting one of the vertices $v_1$ or $v_2$, then Alice responds by choosing the vertex at distance $4$ from the vertex chosen by Bob. With this move, the vertex at distance $2$ from both vertices that have been selected is also forced to be chosen in round $1$. Thus, $V(P_n)-C^1$ is an independent set, and the game will last two rounds. In the case when Bob's first choice is the vertex $v_3$,  Alice selects the vertex $v_5$ in her next move, and then also $v_1$ is forced to be colored with $1$ in the first round. Thus, $V(P_n)-C^1$ is again an independent set, and we infer that $\cigB(P_6)=2$, and $\cigBA (P_6)=2$. 
In the case when $n\geq 7$, Alice cannot enforce that only two rounds will be played in a B-independence and a BA-independence coloring game on $P_n$. Notably, if Bob selects $v_4$ as his first move, then as his second move he can either select $v_1$ or $v_7$, by which two adjacent vertices (lying between the first two vertices selected by Bob) are protected after the first round, hence two more colors will be needed.  Thus $\cigB(P_n)\geq 3$ and $\cigBA (P_n)\geq 3$. By Theorem~\ref{prop2}, we have $\cigB(P_n)\leq 3$ and $\cigBA (P_n)\leq 3$, therefore $\cigB(P_n)=\cigBA (P_n)=3$ for every $n\geq 7$. \qed
\end{proof}

From the proof of Proposition~\ref{prp:_pathsB} one can easily infer that if $P_7$ is an induced subgraph of a graph $G$, Bob can ensure by using the strategy presented in the proof that the game will last at least three rounds. The same conclusion holds for all four versions of the independence coloring game.

\begin{corollary}
\label{cor:paths}
If $G$ is a graph, which contains a path $P_7$ as an induced subgraph, then $\cig(G)\ge 3$. 
\end{corollary}

Next, we consider cycles $C_n$. Clearly, $\cig(C_3)=3$, and $\cig(C_4)=2$ (holds for any type of the independence coloring game). If $n=5$, the first round has two moves in any of the games, and two adjacent vertices are protected after the first round. This yields $\cig(C_5)\geq 3$, and combining with Theorem~\ref{prop2} we get $\cig(C_5)=3$. There is a distinction between different versions of the games in the cycle $C_6$, and so two separate formulas are needed. 

\begin{proposition}
\label{prop:cikliA}
If $C_n$ is the cycle of length $n\geq 3$, then
\begin{displaymath}
\cigA (C_n) =\cigAB (C_n) = \left\{ \begin{array}{ll}
2, & \textrm{$n=4$,}\\
3, & \textrm{otherwise.}
\end{array} \right.
\end{displaymath}
\end{proposition}
\begin{proof}
By the above observations, we need to consider the case when $n\geq 6$. In both versions of the game, Alice has the first move. The response of Bob in his first move is to select a vertex at distance $3$ from the vertex chosen by Alice. Following similar arguments as in the proof for paths $P_n$ when $n\geq 6$, we infer $\cigA(C_n)=\cigAB(C_n)=3$ for any $n\geq 6$. \qed
\end{proof}

\begin{proposition}
\label{prop:cikliB}
If $C_n$ is a cycle of length $n\geq 3$, then
\begin{displaymath}
\cigB (C_n) =\cigBA (C_n) = \left\{ \begin{array}{ll}
2, & \textrm{$n=4,6$,}\\
3, & \textrm{otherwise.}
\end{array} \right.
\end{displaymath}
\end{proposition}
\begin{proof}
First, consider the case when $n=6$. Bob starts a game on $C_6$, and Alice selects in her first move a vertex at distance $2$ from the vertex chosen by Bob. In this way, the vertex at distance $2$ from both chosen vertices is forced to be colored in the first round. It follows that $\cigB(C_6)=\cigBA(C_6)=2$. Finally, when $n\geq 7$, the proof goes along a similar way as in Proposition~\ref{prp:_pathsB}, reaching $\cigB(C_n)=\cigBA(C_n)=3$. \qed
\end{proof}

%%%%%%%%%%%%%%%%%%%%%%%%%%%%%%%%%%%%%%%%%%%%%%%%%
%%%%%%%%%%%%%%%%%%%%%%%%%%%%%%%%%%%%%%%%%%%%%%%
\section{Extremal families}
\label{sec:extremal}
%%%%%%%%%%%%%%%%%%%%%%%%%%%%%%%%%%%%%%%%%%%%%%%%%

We start this section by characterizing the graphs that attain value $2$ in different versions of the independence coloring games. Note that $\cig(G)=1$ if and only if $G$ is edgeless, hence this is the smallest non-trivial case to be considered.

\begin{theorem} \label{izrek_dva}
If $G$ is a connected graph with at least one edge,
then the following statements are equivalent:
\begin{enumerate}[(1)]
\item $\cigA(G) =2$;
\item $\cigAB(G)=2$;
\item $G$ is a bipartite graph with the bipartition $V(G)=(X_1,X_2)$,
and there exists an $i\in\{1,2\}$ and a vertex $x$ in $X_i$, which is adjacent to all vertices from $X_j$, where $\{i,j\}=\{1,2\}$. \end{enumerate}
\end{theorem}

\begin{proof}
First, we prove that if $G$ is a graph that obeys (1) or (2), then (3) holds.
Since $G$ has edges, $\cigA(G)=2$ (respectively, $\cigAB(G)=2$) implies $\chi (G)=2$, and so $G$ is a bipartite graph. Let $(X_1,X_2)$ be the bipartition of $V(G)$.
Since in both, an A-independence and an AB-independence coloring game, Alice starts, we may assume that (any) one of these games is played, and that Alice selected a vertex $x\in X_1$. Suppose that $x$ is not adjacent to all vertices of $X_2$. Then, since $G$ is connected, there exists a vertex $y\in X_2$, which is at distance $3$ from $x$. If Bob chooses $y$ as his second move, then two vertices on a shortest path from $x$ to $y$ are both protected, and are adjacent. Hence, these two vertices will receive different colors, which are also different from $1$, a contradiction with the assumption that the ($A$- or $AB$-) independence game chromatic number is $2$. Hence, $x$ must be adjacent to all vertices of $X_2$, and so the statement (3) is true.

Now, assume that (3) holds, and let $x\in X_1$ be a vertex adjacent to all vertices of $X_2$.  Consider a game in which Alice has the first move. Alice starts the game by choosing $x$, which is adjacent to all vertices of $X_2$, hence no vertex from $X_2$ can be colored by color $1$. In addition, all vertices from $X_1$ will be chosen in the first round. In the second round, only the vertices of $X_2$ remain uncolored, and since they form an independent set, they will all be chosen in round $2$ regardless of who starts the round. Thus, $\cigA(G)=2$ and $\cigAB(G)=2$, which are the statements (1) and (2), respectively. \qed
\end{proof}

\begin{theorem} \label{izrek_dvaB}
If $G$ is a connected graph with at least one edge,
then the following statements are equivalent:
\begin{enumerate}[(1)]
\item $\cigB(G) =2$;
\item $\cigBA(G)=2$;
\item $G$ is a bipartite graph with the bipartition $V(G)=(X_1,X_2)$
such that for each $i\in\{1,2\}$ and for every vertex $x$ in $X_i$ there exist a vertex $y\in X_i$, such that $N(x)\cup N(y)=X_j$, where $\{i,j\}=\{1,2\}$. \end{enumerate}
\end{theorem}
\begin{proof}
First, we prove that if $G$ is a graph that obeys (1) or (2), then (3) holds.
Since $G$ has edges, $\cigB(G)=2$ (respectively, $\cigBA(G)=2$) implies $\chi (G)=2$, and so $G$ is a bipartite graph. Let $(X_1,X_2)$ be the bipartition of $V(G)$. For the purpose of getting a contradiction, assume that (3) does not hold. Since a game is played in which Bob starts, let the first move of Bob be choosing a vertex $x\in X_i$. Note that the response of Alice is not in $X_j$, where $j\ne i$. Indeed, if Alice chose $z\in X_j$, then the distance between $x$ and $z$ is odd (and bigger than $1$, since $x$ and $z$ are not adjacent). Hence, even if $d(x,y)>3$, Bob can achieve in the next move that there will be two vertices colored by $1$, which are at distance $3$. This immediately implies that the number of rounds will be greater than $2$, which is a contradiction. Thus, Alice's response to the first move of Bob is to choose a vertex $y$, which is in $X_i$. By the assumption, $N(x)\cup N(y)\ne X_j$, hence, since $G$ is connected, there exists a vertex $z\in X_j$, which is at distance $3$ from $x$ or from $y$. The same argument as above gives a contradiction. We infer that (3) holds, that is, for each set $X_i$ of the bipartition and for every vertex $x$ of $X_i$ there is a vertex $y\in X_i$ such that all vertices in $X_j$, where $j\ne i$, are adjacent to $x$ or $y$.

The other direction, that (3) implies (1) and (2), can be proved by essentially the same arguments as above. Note that condition (3) implies that for any first move of Bob there exists a response of Alice, by which all vertices of one of the sets $X_i$ are protected, and all vertices of the other set are forced to be given color $1$. Thus, in round $2$ only an independent set remains uncolored, and regardless of who starts round $2$, this means that with this round the game finishes. Hence, (1) and (2) readily follow. \qed
\end{proof}

We continue by presenting a large family of cubic graphs, which attain the upper bound in Theorem~\ref{prop2}.

Let $H$ be an arbitrary graph in which there are two non-adjacent vertices, $v_1$ and $v_{14}$, with $\deg(v_1)=\deg(v_{14})=2$, and all other vertices of $H$ have degree $3$. The graph $G_H$ is obtained from $H$ by adding another $12$ vertices $v_2,\ldots,v_{13}$ such that $v_2,v_3,v_4$ and $v_5$ induce the diamond with $v_2$ and $v_5$ non-adjacent, $v_6,v_7,v_8$ and $v_9$ induce the diamond with $v_6$ and $v_9$ non-adjacent, $v_{10},v_{11},v_{12}$ and $v_{13}$ induce the diamond with $v_{10}$ and $v_{13}$ non-adjacent, and there are also edges $v_1v_2,v_5v_6,v_9v_{10}$ and $v_{13}v_{14}$. See Fig.~\ref{cubic}. Clearly, all graphs in the family ${\cal G}=\{G_H\,|\,H\textit{ has two vertices of degree 2, and all other vertices have degree 3}\}$ are cubic graphs.

\begin{figure}[h]
\begin{center}
\begin{tikzpicture}%[scale=0.95, style=thick]
\def\vr{3pt}
\def\len{1}

\coordinate(v8) at (0,1);
\coordinate (v7) at (0,3);
\coordinate (v6) at (-1,2);
\coordinate (v9) at (1,2);

\coordinate(v5) at (-3,0);
\coordinate (v3) at (-4,-1);
\coordinate (v4) at (-2,-1);
\coordinate (v2) at (-3,-2);
\coordinate (v1) at (-1,-4);

\coordinate(v10) at (3,0);
\coordinate (v11) at (4,-1);
\coordinate (v12) at (2,-1);
\coordinate (v13) at (3,-2);
\coordinate (v14) at (1,-4);
\draw(0,-4)node{$H$};

\draw(v1)[fill=black] circle(\vr);
\draw(v2)[fill=black] circle(\vr);
\draw(v3)[fill=black] circle(\vr);
\draw(v4)[fill=black] circle(\vr);
\draw(v5)[fill=black] circle(\vr);
\draw(v6)[fill=black] circle(\vr);
\draw(v7)[fill=black] circle(\vr);
\draw(v8)[fill=black] circle(\vr);
\draw(v9)[fill=black] circle(\vr);
\draw(v10)[fill=black] circle(\vr);
\draw(v11)[fill=black] circle(\vr);
\draw(v12)[fill=black] circle(\vr);
\draw(v13)[fill=black] circle(\vr);
\draw(v14)[fill=black] circle(\vr);

\draw (v1) -- (v2);
\draw (v2) -- (v3);
\draw (v2) -- (v4);
\draw (v3) -- (v4);
\draw (v4) -- (v5);
\draw (v3) -- (v5);
\draw (v5) -- (v6);
\draw (v6) -- (v7);
\draw (v6) -- (v8);
\draw (v7) -- (v8);
\draw (v7) -- (v9);
\draw (v8) -- (v9);
\draw (v9) -- (v10);
\draw (v10) -- (v11);
\draw (v10) -- (v12);
\draw (v11) -- (v12);
\draw (v11) -- (v13);
\draw (v12) -- (v13);
\draw (v13) -- (v14);

\draw(v1)node[left]{\footnotesize{$v_1$}};
\draw(v2)node[left]{\footnotesize{$v_2$}};
\draw(v3)node[left]{\footnotesize{$v_3$}};
\draw(v4)node[right]{\footnotesize{$v_4$}};
\draw(v5)node[left]{\footnotesize{$v_5$}};
\draw(v6)node[left]{\footnotesize{$v_6$}};
\draw(v7)node[above]{\footnotesize{$v_7$}};
\draw(v8)node[below]{\footnotesize{$v_8$}};
\draw(v9)node[right]{\footnotesize{$v_9$}};
\draw(v10)node[right]{\footnotesize{$v_{10}$}};
\draw(v11)node[right]{\footnotesize{$v_{11}$}};
\draw(v12)node[left]{\footnotesize{$v_{12}$}};
\draw(v13)node[right]{\footnotesize{$v_{13}$}};
\draw(v14)node[right]{\footnotesize{$v_{14}$}};
\draw[dashed] (0,-4) circle (1cm);
\end{tikzpicture}
\end{center}
\caption{A cubic graph $G_H$ with $\cig(G_H)=4$.} 
\label{cubic}
\end{figure}
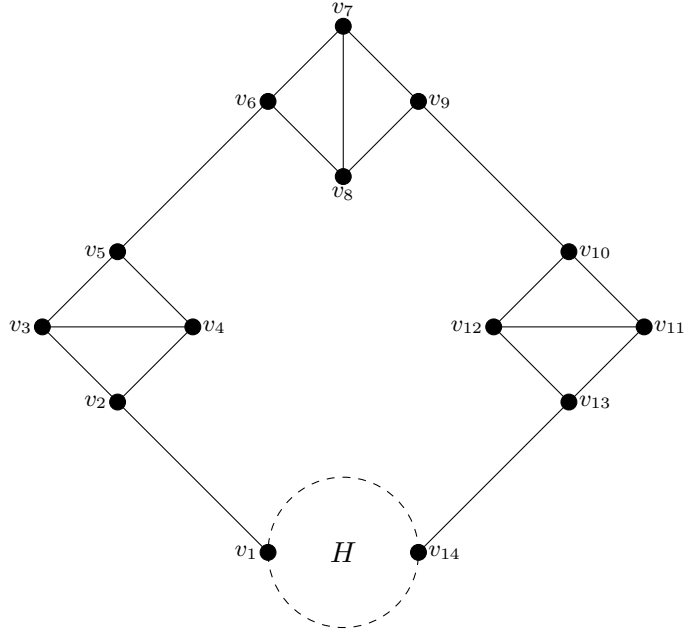

\begin{proposition}
\label{prp:cubic}
For every graph $G_H$ from the family $\mathcal{G}$ (all graphs of which are cubic), we have $\chi_{ig}(G_H)=4$.
\end{proposition}
\begin{proof}
Let Alice and Bob play an independence coloring game on $G_H\in \mathcal{G}$ (see Fig.~\ref{cubic}), and consider the first round. If Alice starts the game and chooses $v_k$ for any $k\in\{1,2,\ldots , 10\}$ in her first move, then Bob chooses the vertex $v_{k+4}$ in his next move, and if Alice chooses $v_k$ for any $k\in\{11,12,13,14\}$ in her first move, then Bob colors the vertex $v_{k-4}$ in his first move. Namely, if $k\in\{1,2,5,6,9,10\}$, then after these two moves the vertices $v_{k+1},v_{k+2},v_{k+3}$ form a triangle, and can no longer receive color $1$. Similarly, if $k\in\{13,14\}$, then the vertices $v_{k-3},v_{k-2},v_{k-1}$ form a triangle, and cannot be chosen in the first round. In this way, the game ends using four colors, since three colors are needed to color $C_3$ and since $G_H$ is a cubic graph, by Theorem~\ref{prop2}, $\chi_{ig}(G_H)\leq \Delta(G_H)+1=4$. Otherwise, if $k\in \{3,4,7,8,11,12\}$, vertices of a path on $6$ vertices are protected after the first two moves. More precisely, if $k\in\{3,7\}$, the vertices $v_{k-1},v_{k+1},v_{k+2},v_{k+3},v_{k+5}$, and $v_{k+6}$ are protected, if $k\in\{4,8\}$ the vertices $v_{k-2},v_{k-1},v_{k+1},v_{k+2},v_{k+3}$, and $v_{k+5}$ are protected, if $k=11$, the vertices $v_{k-5},v_{k-3},v_{k-2},v_{k-1},v_{k+1}$, and $v_{k+2}$ are protected, and if $k=12$, the same holds for the vertices $v_{k-6},v_{k-5},v_{k-3},v_{k-2},v_{k-1}$, and $v_{k+1}$. In fact, in each of these cases, the protected path can be extended by at least one vertex, so that at the end of round $1$ a path on at least $7$ vertices is protected. (For instance, if $k=3$, then the vertices $v_2,v_4,v_5,v_6,v_8$  and $v_9$, which induce a path $P$ on $6$ vertices, are protected. Now, if $v_1$ is played by Alice in her second move, then Bob can play $v_{11}$ and protect also $v_{10},v_{12}$ and $v_{13}$, thus extending the protected path to have $9$ vertices. On the other hand, if $v_1$ is not played by Alice as her second move, then Bob can ensure that after his second move $v_1$ is also protected, which together with $P$ gives a protected path on $7$ vertices. Other cases, when $k\in \{4,7,8,11,12\}$, can be dealt with in a similar way, so that a path on $6$ vertices, which is protected after the first two moves, is extended by at least one vertex yielding a protected $P_7$ after the first round.) Hence, the graph $G_H-C^1$ obtained from $G$ after deleting the vertices selected in the first round contains $P_7$ as an induced subgraph. By Corollary~\ref{cor:paths}, $\cig(G_H-C^1)\ge 3$, for all versions of the independence game chromatic numbers, therefore at least three more colors are needed to finish the game on $G_H$. Again, by Theorem~\ref{prop2}, $\chi_{ig}(G_H)\leq 4=\Delta(G_H)+1$, thus also in this case the game ends using four colors. 

Finally, if Alice starts the game by playing a vertex in $V(H)-\{v_1,v_{14}\}$ or if it is Bob who starts the game, then Bob chooses the vertex $v_6$ as his first move. In his next move, Bob chooses either $v_2$ or $v_{10}$, whichever is possible (according to the first move of Alice preceding this move of Bob). Note that after this move, a triangle is protected. Thus, we infer that also in this case the game lasts four rounds. We conclude that $\chi_{ig} (G_H)=4$. \qed
\end{proof}

We finish the section by studying the A-independence coloring game on the well-known split graphs. As it turns out, the number of moves in such a game equals the clique number, $\omega(G)$, of a split graph $G$, which yields another class of graphs that attain the lower bound in Theorem~\ref{prop2} for one of the independence game chromatic numbers. 

A graph is a \emph{split graph} if its vertex set can be partitioned into an independent set $I$ and a clique $C$. 
Note that in any split graph, one of the following holds \cite{hammer-81}:
\begin{enumerate}[(i)]
\item There exists a vertex $x\in I$ such that $C\cup\{x\}$ induces a complete graph. In this case, $C\cup\{x\}$ induces a maximum clique and $I$ is a maximum independent set.
\item There exists a vertex $x\in C$ such that $I\cup\{x\}$ is independent. In this case, $I\cup\{x\}$ is a maximum independent set and $C$ induces a maximum clique.
\item $C$ induces a maximal clique and $I$ is a maximal independent set. In this case, $G$ has a unique partition $(C,I)$ into a clique and an independent set, $C$ is the maximum clique, and $I$ is the maximum independent set.
\end{enumerate}

\begin{theorem}
\label{thm:split}
If $G$ is a split graph, then $\cigA(G)=\omega(G)$.
\end{theorem}

\begin{proof}
If $G$ is a (split) graph, then $\cigA(G)\geq\chi (G)\geq \omega (G)$.
Therefore, we need to prove that $\cigA(G)\leq \omega (G)$. From \cite{hammer-81} we know that $G$ can be partitioned into an independent set $I$ and a clique $C$ such that one of the three possibilities listed above holds.

First, suppose that the (i) is true, that is, there exists a vertex $x\in I$ such that $C\cup\{x\}$ induces a complete graph. Now, assume that Alice and Bob play an A-independence coloring game on $G$. The strategy of Alice is to choose in her first move the vertex $x$. Then all vertices from $C$ are protected, since $C\cup\{x\}$ is a complete graph. On the other hand, all vertices from $I$ will receive color $1$, since $I$ is an independent set and $x\in I$. Since all vertices from $C\cup\{x\}$ will receive pairwise different colors, the total number of rounds will be $|C|+1$, which is the size of the maximum clique. Thus, $\cigA(G)\leq \omega (G)$. The case when (ii) holds is essentially the same as case (i). Again, Alice will start with the vertex $x$, only that now $x\in C$. However, all vertices of $I$ will be chosen in the first round, which leaves only $\omega(G)-1$ vertices for the next rounds. Hence, $\cigA(G)\leq \omega (G)$.

In the last possibility when (iii) holds in $G$, $C$ is the maximum clique, and $I$ is the maximum independent set. Since $C$ is a maximal clique, no vertex of $I$ is adjacent to all vertices of $C$. Each round of the game is started by Alice, and she chooses a (previously unchosen) vertex $z$ from $C$. The rest of the round is fixed up to the order of chosen vertices, notably, Alice and Bob will have to choose exactly the vertices of $I$ that are not adjacent to $z$. Since no vertex of $I$ is adjacent to all vertices of $C$, the game will last $|C|$ rounds, which is equal $\omega (G)$. Thus, also in this case, $\cigA(G)\leq \omega (G)$, which completes the argument. \qed
\end{proof}

\begin{remark}
Note that the analogues of Theorem~\ref{thm:split} do not hold on all split graphs for the AB-independence, the B-independence and the BA-independence coloring games. Suppose that Alice and Bob play one of the mentioned games on a split graph $G$ and that we are in case (iii) ($C$ is the maximum clique and $I$ the maximum independent set in $G$). If Bob starts the round $k$ for some $k\in\{2,\ldots ,\omega(G)\}$ and chooses a vertex $u\in I$ such that $u$ is adjacent to all vertices in $C$, which have not yet been colored (if such a vertex $u$ exists in round $k$), then no vertex from $C$ can be colored with color $k$. This means that at least $\omega(G)+1$ colors will be needed to finish the game on $G$, since $\omega(G)$ vertices of $C$ must clearly receive $\omega(G)$ different colors and none of them receives color $k$.
\end{remark}

%%%%%%%%%%%%%
%T  R  E  E  S
%%%%%%%%%%%%%
%%%%%%%%%%%%%%%%%%%%%%%%%%%%%%%%%%%%%%%%

\section{Trees}
\label{sec:trees}

The following general result can be useful in proving bounds for the independence game chromatic numbers. We introduce another variant that is closely related to the independence game chromatic number. The only difference from the independence coloring game is that one of the players is allowed to skip moves. 

Suppose that Alice and Bob are playing an independence coloring game with just one change, by which Alice is allowed (but not required) to skip any number of moves. That is, in each round Alice and Bob are taking turns by selecting a previously unselected vertex of a graph $G$ in such a way that an independent set is created. However, whenever she wants, Alice can choose to pass a move and it is Bob's turn again. We call such a game in which Alice is allowed to skip any number of moves an {\em Alice-skip independence coloring game}; we also assume that Alice starts this game (yet, she may decide to skip the first move as well), and each new round is started by the player who did not end the previous round.  The resulting number of rounds in an Alice-skip independence coloring game in which Alice and Bob are playing optimally according to their goals, is the {\em Alice-skip independence game chromatic number}, $\cigas(G)$, of $G$. Clearly, since Alice is allowed to play without skipping any move in the Alice-skip independence coloring game, $\cigas(G)\le \cigAB(G)$ in an arbitrary graph $G$. Also, $\cigas(G)\le \cigBA(G)$, since Alice can choose to pass only the first move and then a BA-independence coloring game is played. It is also easy to see that $\cigas(G)\le \cigB(G)$, since Alice may choose to pass the first move of any round in which it would be her turn, and so a B-independence coloring game is played.

\begin{lemma}
\label{lem:Alice-skip}
Let Alice and Bob play an Alice-skip independence coloring game on a graph $G$. If Bob can ensure that after the first round there is a component $H$ of the graph $G-C^1$ (where $C^1$ is the set of vertices selected in the first round), then  $\cigas(G)\ge \cigas(H)+1$. In addition, $\cigAB(G)\ge \cigas(H)+1$, $\cigBA(G)\ge \cigas(H)+1$, and $\cigB(G)\ge \cigas(H)+1$. 
\end{lemma}
\proof
Suppose that after the first round of an Alice-skip independence coloring game, the set $C^1$ has been selected, and one of the components of the graph $G-C^1$ is $H$. Then, in the next rounds Bob imagines to play an Alice-skip independence coloring game in $H$. That is, in each round Bob plays only in the subgraph $H$ as long as possible by using his optimal strategy for the Alice-skip independence coloring game in $H$. If this is no longer possible (that is, after the vertices selected in $H$ in a given round form a maximal independent  set in $H$), then Bob plays in other components of $G-C^1$ in any way. Clearly, in the end at least $\cigas(H)$ rounds will be played in the graph $G-C^1$, hence altogether there will be at least $1+\cigas(H)$ rounds. The last sentence follows from the fact that  $\cigas(G)\le \cigAB(G)$,  $\cigas(G)\le \cigBA(G)$, and  $\cigas(G)\le \cigB(G)$ in any graph $G$.
\qed

Recall that a {\em perfect $n$-ary tree} is a rooted tree in which all leaves have the same depth and all interior vertices have $n$ children. We will use the notation $T(n,d)$ for the perfect $n$-ary tree of depth $d$.  We label the vertices of such a tree in the following way. The root is labeled by the empty label $()$, while the labels of the $n$ children of an arbitrary interior vertex in the tree, which is labeled by $\ell$ are obtained from $\ell$ by adding an integer from $\{1,2,\ldots ,n\}$ to the right side of label $\ell$. More specifically, $1$ is added to the right of $\ell$ for the left-most child, the second left-most child is labeled by $(\ell 2)$, and the right-most child gets the label $(\ell n)$. This means that the vertices on the $k$th level (where the root is on the $0$th level) of $T(n,d)$ are the $k$-tuples of integers from $\{1,2,\ldots ,n\}$. The left-most leaf is labeled by $(1\ldots 1)$ and the right-most leaf by $(n\ldots n)$. In Fig.~\ref{fig:tree}, $T(3,3)$ is depicted together with its labels. If $T$ is a copy of the tree $T(n,d)$ and $(\ell)$ is the label of a vertex of $T$, then by $T(\ell)$ we denote the subtree of $T$ rooted at $(\ell)$. Note that $T(\ell)$ contains the vertices whose label starts with $\ell$. 

\begin{figure}[h]
\begin{center}
\begin{tikzpicture}%[scale=0.95, style=thick]
\def\vr{2pt}
\def\len{1}

\coordinate(v0) at (0,0);
\coordinate (v1) at (-4.5,-1.5);
\coordinate (v2) at (0,-1.5);
\coordinate (v3) at (4.5,-1.5);

\coordinate (v11) at (-6,-3);
\coordinate (v21) at (-4.5,-3);
\coordinate (v31) at (-3,-3);
\coordinate (v12) at (-1.5,-3);
\coordinate (v22) at (-0,-3);
\coordinate (v32) at (1.5,-3);
\coordinate (v13) at (3,-3);
\coordinate (v23) at (4.5,-3);
\coordinate (v33) at (6,-3);

\coordinate (v111) at (-6.5,-4.5);
\coordinate (v211) at (-6,-4.5);
%\coordinate (x211) at (-6,-5);
\coordinate (v311) at (-5.5,-4.5);
\coordinate (v121) at (-5,-4.5);
\coordinate (v221) at (-4.5,-4.5);
\coordinate (v321) at (-4,-4.5);
\coordinate (v131) at (-3.5,-4.5);
\coordinate (v231) at (-3,-4.5);
\coordinate (v331) at (-2.5,-4.5);

\coordinate (v112) at (-2,-4.5);
\coordinate (v212) at (-1.5,-4.5);
\coordinate (v312) at (-1,-4.5);
\coordinate (v122) at (-0.5,-4.5);
\coordinate (v222) at (-0,-4.5);
\coordinate (v322) at (0.5,-4.5);
\coordinate (v132) at (1,-4.5);
\coordinate (v232) at (1.5,-4.5);
\coordinate (v332) at (2,-4.5);

\coordinate (v113) at (2.5,-4.5);
\coordinate (v213) at (3,-4.5);
\coordinate (v313) at (3.5,-4.5);
\coordinate (v123) at (4,-4.5);
\coordinate (v223) at (4.5,-4.5);
\coordinate (v323) at (5,-4.5);
\coordinate (v133) at (5.5,-4.5);
\coordinate (v233) at (6,-4.5);
\coordinate (v333) at (6.5,-4.5);

\draw(v111)[fill=black] circle(\vr);
\draw(v211)[fill=black] circle(\vr);
%\draw(x211)[];
\draw(v311)[fill=black] circle(\vr);
\draw(v121)[fill=black] circle(\vr);
\draw(v221)[fill=black] circle(\vr);
\draw(v321)[fill=black] circle(\vr);
\draw(v131)[fill=black] circle(\vr);
\draw(v231)[fill=black] circle(\vr);
\draw(v331)[fill=black] circle(\vr);

\draw(v112)[fill=black] circle(\vr);
\draw(v212)[fill=black] circle(\vr);
\draw(v312)[fill=black] circle(\vr);
\draw(v122)[fill=black] circle(\vr);
\draw(v222)[fill=black] circle(\vr);
\draw(v322)[fill=black] circle(\vr);
\draw(v132)[fill=black] circle(\vr);
\draw(v232)[fill=black] circle(\vr);
\draw(v332)[fill=black] circle(\vr);

\draw(v113)[fill=black] circle(\vr);
\draw(v213)[fill=black] circle(\vr);
\draw(v313)[fill=black] circle(\vr);
\draw(v123)[fill=black] circle(\vr);
\draw(v223)[fill=black] circle(\vr);
\draw(v323)[fill=black] circle(\vr);
\draw(v133)[fill=black] circle(\vr);
\draw(v233)[fill=black] circle(\vr);
\draw(v333)[fill=black] circle(\vr);

\draw(v0)[fill=black] circle(\vr);
\draw(v1)[fill=black] circle(\vr);
\draw(v2)[fill=black] circle(\vr);
\draw(v3)[fill=black] circle(\vr);

\draw(v11)[fill=black] circle(\vr);
\draw(v21)[fill=black] circle(\vr);
\draw(v31)[fill=black] circle(\vr);
\draw(v12)[fill=black] circle(\vr);
\draw(v22)[fill=black] circle(\vr);
\draw(v32)[fill=black] circle(\vr);
\draw(v13)[fill=black] circle(\vr);
\draw(v23)[fill=black] circle(\vr);
\draw(v33)[fill=black] circle(\vr);

\draw (v0) -- (v1) -- (v11) -- (v111);  
\draw (v0) -- (v2);
\draw (v0) -- (v3);
\draw (v1) -- (v11);
\draw (v1) -- (v21);
\draw (v1) -- (v31);
\draw (v2) -- (v12);
\draw (v2) -- (v22);
\draw (v2) -- (v32);
\draw (v3) -- (v13);
\draw (v3) -- (v23);
\draw (v3) -- (v33);
\draw (v11) -- (v111);
\draw (v11) -- (v211);
\draw (v11) -- (v311);
\draw (v21) -- (v121);
\draw (v21) -- (v221);
\draw (v21) -- (v321);
\draw (v31) -- (v131);
\draw (v31) -- (v231);
\draw (v31) -- (v331);

\draw (v12) -- (v112);
\draw (v12) -- (v212);
\draw (v12) -- (v312);
\draw (v22) -- (v122);
\draw (v22) -- (v222);
\draw (v22) -- (v322);
\draw (v32) -- (v132);
\draw (v32) -- (v232);
\draw (v32) -- (v332);

\draw (v13) -- (v113);
\draw (v13) -- (v213);
\draw (v13) -- (v313);
\draw (v23) -- (v123);
\draw (v23) -- (v223);
\draw (v23) -- (v323);
\draw (v33) -- (v133);
\draw (v33) -- (v233);
\draw (v33) -- (v333);

\draw(v0)node[above]{\footnotesize{$()$}};
\draw(v1)node[left]{\footnotesize{$(1)$}};
\draw(v2)node[left]{\footnotesize{$(2)$}};
\draw(v3)node[right]{\footnotesize{$(3)$}};
\draw(v11)node[left]{\footnotesize{$(11)$}};
\draw(v21)node[left]{\footnotesize{$(12)$}};
\draw(v31)node[left]{\footnotesize{$(13)$}};
\draw(v12)node[left]{\footnotesize{$(21)$}};
\draw(v22)node[left]{\footnotesize{$(22)$}};
\draw(v32)node[left]{\footnotesize{$(23)$}};
\draw(v13)node[left]{\footnotesize{$(31)$}};
\draw(v23)node[left]{\footnotesize{$(32)$}};
\draw(v33)node[left]{\footnotesize{$(33)$}};

\draw(v111)node[below]{\tiny{$(111)$}};
\draw(v211)node[above]{\tiny{$(112)$}};
\draw(v311)node[below]{\tiny{$(113)$}};
\draw(v121)node[above]{\tiny{$(121)$}};
\draw(v221)node[below]{\tiny{$(122)$}};
\draw(v321)node[above]{\tiny{$(123)$}};
\draw(v131)node[below]{\tiny{$(131)$}};
\draw(v231)node[above]{\tiny{$(132)$}};
\draw(v331)node[below]{\tiny{$(133)$}};

\draw(v112)node[above]{\tiny{$(211)$}};
\draw(v212)node[below]{\tiny{$(212)$}};
\draw(v312)node[above]{\tiny{$(213)$}};
\draw(v122)node[below]{\tiny{$(221)$}};
\draw(v222)node[above]{\tiny{$(222)$}};
\draw(v322)node[below]{\tiny{$(223)$}};
\draw(v132)node[above]{\tiny{$(231)$}};
\draw(v232)node[below]{\tiny{$(232)$}};
\draw(v332)node[above]{\tiny{$(233)$}};

\draw(v113)node[below]{\tiny{$(311)$}};
\draw(v213)node[above]{\tiny{$(312)$}};
\draw(v313)node[below]{\tiny{$(313)$}};
\draw(v123)node[above]{\tiny{$(321)$}};
\draw(v223)node[below]{\tiny{$(322)$}};
\draw(v323)node[above]{\tiny{$(323)$}};
\draw(v133)node[below]{\tiny{$(331)$}};
\draw(v233)node[above]{\tiny{$(332)$}};
\draw(v333)node[below]{\tiny{$(333)$}};

\end{tikzpicture}
\end{center}
\caption{Perfect $3$-ary tree of depth $3$, $T(3,3)$.} 
\label{fig:tree}
\end{figure}
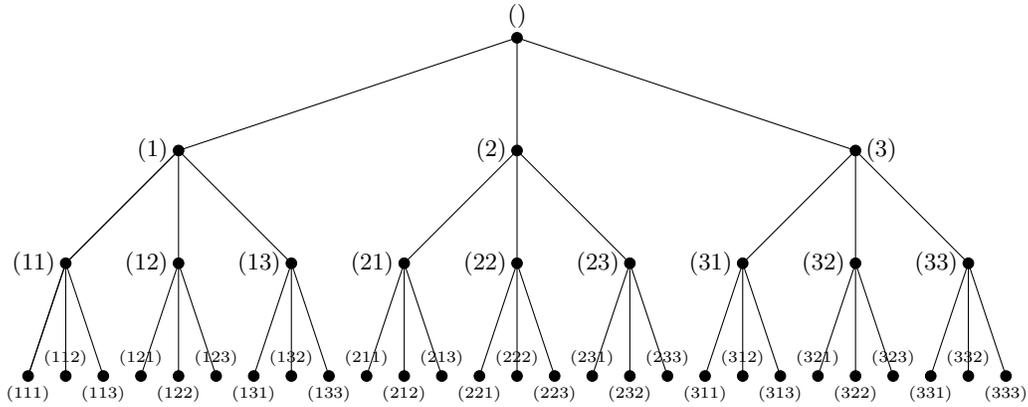

\begin{theorem}
\label{thm:trees}
For any positive integer $k$ there exists a tree $G_k$ such that $\cigas(G_k)\ge k$.   
\end{theorem}

\begin{proof}
Let $G_1$ be the one-vertex graph, and let $G_2$ be the path $P_2$. Clearly, $\cigas(G_1)=1$, and $\cigas(G_2)=2$. Next, let $G_3$ be the perfect binary tree $T(2,3)$. It is easy to see that $\cigas(G_3)\ge 3$. Indeed, Bob can ensure that two vertices at distance $3$ are selected in the first two consecutive moves in an Alice-skip independence coloring game played on $G_3$ by which at least three rounds will be needed in the game.

Now, let $k\ge 4$ be a positive integer. Let $G_k$ be the tree $T(3\cdot2^{k-3}-1,2k-3)$ (for $k=3$, this gives $T(2,3)$ which will be the base case).
Assume that Alice and Bob are playing an Alice-skip independence coloring game on $G_k$.  We will prove that Bob can ensure that after the first round of this game, with $C^1$ as the set of vertices selected in the first round, there is a component of $G-C^1$ that is isomorphic to $G_{k-1}$ (note that $G_{k-1}=T(3\cdot2^{k-4}-1,2k-5)$). By Lemma~\ref{lem:Alice-skip}, $\cigas(G_k)\ge \cigas(G_{k-1})+1$. By using induction on $k$, we infer that $\cigas(G_k)\ge k$.
(As noted above, the basis of induction is settled by $\cigas(G_3)\ge 3$.) 

To simplify the notation, let $n=3\cdot2^{k-3}-1$ denote the degree of non-leaf vertices in $G_k$ and $d=2k-3$ the depth of $G_k$, while $n'=3\cdot2^{k-4}-1$ is the non-leaf degree in $G_{k-1}$ and $d'=2k-5$ is its depth. Note that $n=2n'+1$ and $d=d'+2$. 

Let Alice and Bob play the first round (using color $1$) of an Alice-skip independence coloring game on a tree $T$ isomorphic to $G_k$. Bob will ensure that a subgraph $T'$ isomorphic to $G_{k-1}$ is a component of $T-C^1$, where $C^1$ is the set of vertices chosen in the first round. We start by determining the root of the tree $T'$, which will be protected after the first round, according to Bob's strategy. First, if Alice selects the root $()$ of $T$ as the first move in the game on $T$, then the root of $T'$ will be the  vertex $(1)$. Similarly, if Alice plays in any subtree $T(i)$, but not a neighbor of the root $()$, then Bob will play the root $()$. We may assume without loss of generality that Alice's first move was in the subtree $T(i)$, where $i>1$. Then Bob can choose $(1)$ to be the root of the tree $T'$. Finally, if Alice plays a neighbor of $()$, say $(i)$ with $i>1$, then the root $()$ of $T$ will also be the root of $T'$, while all other vertices of $T'$ will be in $T(1)$. We will consider in full detail only the case when one of the players starts the game by choosing $()$, while the other case can be dealt with in a similar way. 

Since the root $()$ is already colored by $1$, its neighbors, which are the vertices of the first level of $T$, can no longer receive color $1$. In other words, the vertices in the first level of $T$ are now protected. In particular, $(1)$ is protected. Note that $(1)$ is the root of a tree $T'$. The strategy of Bob is to protect vertices of $T'$ one by one, where the choice of which vertex will be added to $T'$ as a new protected vertex is in part based on Alice's moves. Yet, vertices of $T'$ will always be protected in up-to-down order, meaning that a vertex of $T'$ will be protected only if its parent has been protected earlier.  At the end, the resulting protected tree $T'$ should be isomorphic to $G_{k-1}$.  Note that Bob considers only the moves in the subtree $T(1)$, even though Alice may play also outside this subtree; hence Bob imagines an Alice-skip game is played in the subtree $T(1)$.

The start is clear, since the vertex $(1)$ (which is the root of $T'$) is protected after one the first two moves of the game. If Alice does not play her next move in the subtree $T(1)$, then Bob can protect a child of $(1)$, say $(11)$, by playing its child, say $(11n)$. But even if Alice plays in $T(1)$, Bob can protect a child of $(1)$ in his next move. In particular, if Alice plays in the subtree $T(1i)$, then Bob can protect the child $(1j)$, where $j\neq i$, by playing its child, say $(1jn)$. Clearly, after these moves, the root $(1)$ and one of its children, $(1j)$ are added to $T'$ and are protected. 

Now, let us present the strategy of Bob in general. Let $S'$ be the set of vertices that have been protected up to a certain point in the game and thus belong to $T'$, and it is Alice's turn. (Note that by Bob's strategy, the vertices of $S'$ induce a subtree of $T'$ containing the root $(1)$.) 
Suppose that Alice selects a vertex $x$ in $T(1)$. Then $x$ is in the subtree $T(1v_2\ldots v_t)$, where $(1v_2\ldots v_t)$ is the closest predecessor to $x$ that is protected (clearly, such a predecessor exists, since $(1)\in V(T(1))$ is protected). Now, there are two possibilities. If in $S'$ there are already $n'$ (protected) children of $(1v_2\ldots v_t)$, then Bob chooses any unprotected vertex $y\in V(T(1))$ whose parent $w$ is already in $S'$ and $w$ does not yet have $n'$ protected children, and is as close as possible to the root; Bob then protects $y$ in his next move. Otherwise, if there are less than $n'$ children of $(1v_2\ldots v_t)$ in $S'$, then Bob chooses an unprotected child $z$ of $(1v_2\ldots v_t)$, and protects it in his next move by selecting a child of $z$. Moreover, he can choose $z$ to be the child of $(1v_2\ldots v_t)$ such that no vertex of $T(z)$ has been chosen earlier in the game. Indeed, $(1v_2\ldots v_t)$ is the closest predecessor to $x$ that is protected, and whenever Alice plays in the subtree $T(1v_2\ldots v_t)$ Bob's strategy is to follow with a move in $T(1v_2\ldots v_t)$. More precisely, whenever she selects a vertex in the subtree $T(1v_2\ldots v_tu)$ of $T(1v_2\ldots v_t)$ such that no vertex from $T(1v_2\ldots v_tu)$ has been selected earlier, Bob chooses another subtree $T(1v_2\ldots v_tv)$ of $T(1v_2\ldots v_t)$, where $v\neq u$, such that no vertex from $T(1v_2\ldots v_tv)$ has been selected earlier, and he protects the child $(1v_2\ldots v_tv)$ of $(1v_2\ldots v_t)$ (by selecting, say $(1v_2\ldots v_tvn)$). Since there are $n=2n'+1$ children of $(1v_2\ldots v_t)$, this implies that even if Alice is the first to select a vertex in the subtree $T(1v_2\ldots v_t)$ Bob can ensure that $n'$ children of $(1v_2\ldots v_t)$ are protected and added to $T'$. 

Finally, if Alice skips a move, or if she does not play in the subtree $T(1)$, then until $S'\subsetneq V(T')$ Bob chooses any unprotected vertex $y$ whose parent $w$ is in $S'$ (and is thus protected), $w$ is on the level smaller than $d-1$, but does not yet have $n'$ protected children, and Bob protects $y$ in his next move. 

Since the depth $d'$ of $G_{k-1}$ equals $d-2$, where $d$ is the depth of $G_k$,   it is clear that Bob can protect a subtree $T'$ isomorphic to $G_{k-1}$ in the first round, by which the proof is complete.  
\qed
\end{proof}

As an immediate consequence of Theorem~\ref{thm:trees} we infer
\begin{corollary}
\label{cor:trees}
For any positive integer $k$ there exists a tree $G_k$ such that $\cigAB(G_k)\ge k$, $\cigBA(G_k)\ge k$, and $\cigB(G_k)\ge k$. 
\end{corollary}

\begin{remark}
By using a similar construction as in Theorem~\ref{thm:trees} one can also prove that there exists a tree $T$ such that $\cigA(T)$ can be arbitrarily large. Since the proof methods are analogous to the above, we omit the proof.
\end{remark}

\section{Comparing the independence game chromatic numbers}
\label{sec:compare}

Let $k$ be an arbitrary positive integer and $K_{k,k}$ be the complete bipartite graph with partite sets $A=\{a_1,\ldots ,a_{k}\}$ and $B=\{b_1,\ldots ,b_{k}\}$. Let $M$ be a perfect matching in $K_{k,k}$, say $M=\{a_1b_1,\ldots ,a_{k}b_{k}\}$. Let $G_1^k$ be the graph $K_{k,k}-M$.

\begin{lemma}
\label{lem:G1-A-AB}
$\cigA(G_1^k)\ge k$, and $\cigAB(G_1^k)\ge k$.
\end{lemma}

\begin{proof}
Suppose that Alice and Bob play an A-independence or an AB-independence coloring game on $G_1^k$. In both games Alice has the first move in the first round of the game. Without loss of generality, we may assume that Alice chooses the vertex $a_1$. Bob answers by choosing $b_1$, after which the first round ends (all the remaining vertices are in $N(a_1)\cup N(b_1)$). 
In the next round of the game, it is again Alice's turn (in either of the mentioned games). Clearly, Bob can achieve that each round ends in two moves, by choosing a vertex connected along the edge of $M$ with the vertex chosen by Alice in that round. In this way, the game ends in $k$ rounds. \qed
\end{proof}

Note that $G_1^k$ enjoys the statement (3) in Theorem~\ref{izrek_dvaB}, hence we infer the following result.

\begin{lemma}
\label{lem:G1-B-BA}
$\cigB(G_1^k)=\cigBA(G_1^k)=2$.
\end{lemma}

For an arbitrary positive integer $k$, let $G_2^k$ be the graph obtained from $G_1^k$ by adding a universal vertex $u$ (a vertex adjacent to all vertices of $G_1^k$).

\begin{lemma}
\label{lem:G2-A-BA}
$\cigA(G_2^k)\ge k+1$ and $\cigBA(G_2^k)\ge k+1$.
\end{lemma}
\begin{proof}
First, we prove that $\cigA(G_2^k)\ge k+1$. Suppose that Alice and Bob play an A-independence coloring game on $G_2^k$. Since Alice starts each round of the game,
she may either start a round by playing the vertex $u$ (after which the round is finished) or she chooses a vertex $x$ in $G_1^k$ (after which Bob can finish the round by playing the vertex $y$, incident with the edge $xy\in M$). In any case, with this strategy, Bob can enforce at least $k+1$ rounds in the game.

Suppose that Alice and Bob play an BA-independence coloring game on $G_2^k$. Bob's first move is to choose the vertex $u$. After this move, no other vertex can receive color $1$, since $u$ is a universal vertex. Therefore, the first round of the game is over, and it is Alice's turn. In the same way as before we note that Bob has a strategy to enforce that each further round takes exactly two moves. In this way, $k+1$ rounds will be played in the game on $G_2^k$. Hence, $\cigBA(G_1)\ge k+1$. \qed
\end{proof}

\begin{lemma}
\label{lem:G2-AB-B}
$\cigAB(G_2^k)=\cigB(G_2^k)=3$.
\end{lemma}
\begin{proof}
Since $G_2^k$ is not bipartite, Theorems~\ref{izrek_dva} and~\ref{izrek_dvaB} imply that $\cigAB(G_2^k)>2$ and $\cigB(G_2^k)>2$. 

Suppose that Alice and Bob play an AB-independence coloring game on $G_2^k$. We present the strategy of Alice. She starts the game by choosing the vertex $u$, by which the first round ends. Now, it is Bob's turn. Since $G_2^k-u$ is isomorphic to $G_1^k$, we infer by Lemma~\ref{lem:G1-B-BA} that Alice can enforce that only two rounds will be played in the rest of the game. Thus $\cigAB(G_2^k)=3$. The same strategy works for Alice in a B-independence coloring game, hence $\cigB(G_2^k)=3$. \qed
\end{proof}

For an arbitrary positive integer $k$, we construct the graph $G_3^k$ as follows. The set of vertices of $G_3^k$ partitions into three vertex subsets, $X,Y$ and $Z$, which are all independent sets of vertices in $G_3^k$. The set $X$ consists of only one vertex $x$, which is adjacent precisely to all vertices of $Y$. Next, $Y=\{y_1,y_2,\ldots ,y_{4k}\}$, and $Z$ partitions into $s={{4k}\choose{2k}}$ subsets $Z_1,Z_2,\ldots ,Z_{s}$ with $|Z_i|=2k$ for all $i\in [{{4k}\choose{2k}}]$. For any $(2k)$-subset $Y_i$ in $Y$ there exists a unique set $Z_i$ such that all vertices in $Y_i$ are adjacent to all vertices in $Z_i$.  

Since $G_3^k$ is a bipartite graph with partite sets $X\cup Z$ and $Y$, and $x\in X$ is adjacent to all vertices of $Y$, we directly infer by Theorem~\ref{izrek_dva} the value of the A-independence and AB-independence game chromatic numbers of $G_3^k$.
\begin{lemma}
\label{lem:G3-AB-A}
\label{lemma_G3}
$\cigA(G_3^k)=\cigAB(G_3^k)=2$.
\end{lemma}

The BA-independence game chromatic number of $G_3^k$ needs some extra work. 

\begin{lemma}
\label{lem:G3-BA}
$\cigBA(G_3^k)\le 4$.
\end{lemma}

\begin{proof}
Suppose that Alice and Bob play a BA-independence coloring game on $G_3^k$. The first round of the game is started by Bob. Suppose that during the first round an odd number of vertices in $Y$ have been colored. The number of vertices in $Z$ that are colored in every round of the game is even, since whenever one of the vertices in $Z_i$ is chosen in a round, all $2k$ vertices in a set $Z_i$ will be chosen in that round; hence, assuming that $x$ is not chosen in a given round, an odd number of vertices are chosen in that round, and so the player who starts the next round is different from the one who started the current round. This means that Alice will start the second round. If Alice starts the second round, she chooses $x$, and the game then lasts only three rounds (notably, Alice choosing $x$ implies that no vertex from $Y$ and all the remaining vertices from $Z$ will be chosen in round $2$, hence for round $3$ only an independent set of uncolored vertices remains). 
Clearly, Bob wants to prevent this situation, and wants to achieve that the first round of the game ends with an even number of colored vertices in $Y$.   

Consider the first round of the game. If Bob chooses the vertex $x$ in the first move of the game, then the game ends in two rounds by the same argument as in the previous paragraph. Next, if Bob starts the game by choosing a vertex of $Z$, then Alice chooses $x$, and again the game lasts only two rounds. Therefore, Bob starts the first round of the game by choosing a vertex $y'$ of $Y$. Alice responds by choosing a vertex of $Z$, by which she protects $2k$ vertices in $Y$. Now, Bob must color another vertex in $Y$, otherwise the round ends with only one chosen vertex from $Y$, namely $y'$. (Indeed, if Bob's second move is not in $Y$, then Alice can choose in her second move a vertex of $Z$, which is adjacent to all uncolored vertices in $Y$, and so only one vertex of $Y$ receives color $1$.) Since now two vertices of $Y$ are colored, Alice plays in $Y$. The procedure continues in the same way, that is, Bob must choose a vertex of $Y$ to prevent that an odd number of vertices are chosen from $Y$ in round $1$. In this way, the first round ends with $2k$ vertices from $Y$ that are colored by $1$.
Thus, Bob starts the second round of the game. From the same reason as in the first round of the game, he starts the round by coloring a vertex in $Y$ (which has not been given color $1$ in the first round). Now, Alice chooses to color a vertex in $Z$, which is adjacent to $y'$ and to all yet uncolored vertices in $Y$. In this way, in the second round of the game only one vertex in $Y$ is selected, therefore Alice starts the third round of the game. She colors vertex $x$ as the first move of the third round, and so all vertices of $Z$ are also colored in this round. The game ends in round $4$, hence the strategy of Alice to keep the number of colors bounded by $4$ is realized. \qed
\end{proof}

Finally, we show that the B-independence game chromatic number can be arbitrarily large in $G_3^k$.

\begin{lemma}
\label{lem:G3-B}
$\cigB(G_3^k)\ge k$.
\end{lemma}

\begin{proof}
Let Alice and Bob play a B-independence coloring game on $G_3^k$. Bob is to start each round of the game. The strategy of Bob is to starts each round of the game (except the last one) by choosing a vertex of $Y$, and so preventing $x$ to be selected. The goal of Bob is that as few vertices in $Y$ as possible receive the color of the round. Hence, if Alice responses in $Z$, she helps Bob to realize his plan. Thus, assume that Alice chooses a vertex in $Y$. The next move of Bob is to choose a vertex $z$ in $Z_i$, for some $i\in [s]$. Clearly, by the rules of the game, $z$ is adjacent to $2k$ vertices in $Y$, which are different from the vertices that have been selected. (In this way, Bob protects $2k$ vertices in $Y$). Again, assume that Alice chooses another vertex in $Y$ as her next move. Let $Y_1$ be the set of vertices from $Y$ chosen to that point. In the next move, Bob chooses $z'\in Z_j$, for some $j\ne i$, such that $N(z)\cup N(z')=Y-Y_1$, that is, all vertices from $Y$ that have not been chosen in previous moves are now protected. After the last move of Bob, no vertices from $Y$ can be selected in round $1$, therefore all vertices of the sets $Z_i$, such that there are no edges between $Z_i$ and $Y_1$ must be selected. (It is clear that even if Alice has not  chosen only vertices from $Y$, but has also chosen some vertices from $Z$ the outcome is the same, only that $Y_1$ has fewer vertices). That is, in round $1$ vertices from a set $Y_1$, where $|Y_1|\le 3$, have been selected, and all vertices from the sets $Z_i$ such that $N(Y_1)\cap Z_i=\emptyset$ have also been selected.

The game continues in the same way. Thus, in round $2$, a set $Y_2\subset Y-Y_1$ is selected, where $|Y_2|\le 3$. In addition, in round $2$ exactly the vertices from the sets $Z_i$ are selected with $Z_i$ being adjacent to at least one vertex from $Y_1$ and to no vertex of $Y_2$. More generally, in round $j$ a set $Y_j$ of at most three vertices from $Y$ is selected, and the corresponding (i.e., non-adjacent to $Y_j$ and previously non-selected) vertices from $Z$ are also selected. This implies that after round $k-1$, there are at most $3(k-1)$ vertices in $Y$, which belong to $Y_1\cup\cdots\cup Y_{k-1}$, that were selected. There also exist vertices in $Z$ that have not been selected after round $k-1$, and they are the vertices from the sets $Z_i$ that have at least one neighbor from each of the sets $Y_1,\ldots, Y_{k-1}$ (by definition of the adjacencies in $G_3^k$). In addition, $x$ has not been selected after $k-1$ rounds. We conclude that $\cigB(G_3^k)\ge k$.  \qed
\end{proof}

From the above lemmas, we obtain the following results about possible differences  between pairs of independence coloring game invariants. 

\begin{theorem}
\label{thm:compare}
Let $\overline{\cig}$ and $\widetilde{\cig}$ be any two invariants in $\{\cigA,\cigB,\cigAB,\cigBA\}$. For each ordered pair $(\overline{\cig},\widetilde{\cig})$, except perhaps $(\cigAB,\cigA)$ and $(\cigBA,\cigA)$, there exist graphs $G$ such that the difference $\overline{\cig}(G)-\widetilde{\cig}(G)$ is arbitrarily large. 
\end{theorem}

Lemmas~\ref{lem:G1-A-AB} and~\ref{lem:G1-B-BA} take care of the pairs $(\cigA,\cigB),(\cigA,\cigBA),(\cigAB,\cigB)$ and $(\cigAB,\cigBA)$, Lemmas~\ref{lem:G2-A-BA} and~\ref{lem:G2-AB-B} settle the pairs $(\cigA,\cigAB),(\cigBA,\cigAB)$ and $(\cigBA,\cigB)$, while Lemmas~\ref{lem:G3-AB-A}, \ref{lem:G3-BA} and~\ref{lem:G3-B} take care of all ordered pairs with $\cigB$ as the first entry. Among twelve possible ordered pairs, only $(\cigAB,\cigA)$ and $(\cigBA,\cigA)$ remain unsettled. It is left as an open problem, whether there is a graph such that 
the difference $\cigAB(G)-\cigA(G)$, respectively $\cigBA(G)-\cigA(G)$, is arbitrarily large. 

\section{The independence game chromatic numbers versus the game chromatic number}
\label{sec:othergames}

One of the most important games in graphs is the coloring game. Given a graph $G$ and a fixed set of $k$ colors $C$, the two players, called Alice and Bob, alternately color the vertices of a given graph with Alice going first. When coloring a vertex they are using an arbitrary color from $C$ as long as the proper coloring rule is fulfilled after the move. The goal of Alice is that in the end all vertices of $G$ are colored, while Bob is trying to prevent this. The minimum number of colors needed for Alice to win the coloring game on $G$ is the {\em game chromatic number} of $G$, denoted by $\chi_g(G)$.

A related graph invariant is obtained from the {\em marking game}, played on a graph $G$. It is a game of two players, Alice and Bob, in which the players alternatively mark vertices of a given graph. Given a positive integer $k$, the game proceeds as follows. Alice has the first move. The aim of Bob is to reach a situation where an unmarked vertex has $k$ marked neighbors, while Alice wants to prevent such a situation. The smallest $k$ for which Alice can prevent in any step of the game that an unmarked vertex has $k$ marked neighbors, is the {\em game coloring number}, denoted by $col_g(G)$. Obviously, $\chi_g(G)\leq col_g(G)$.

Faigle et al.~\cite{faigle} proved that the game coloring number as well as the game chromatic number of a forest is at most $4$. In the previous section we proved that the independence game chromatic number of a tree can be arbitrary large. Hence the following is true:

\begin{theorem}
For an arbitrary positive integer $k$ there exists a graph $G$ such that $\chi_{ig}(G)-col_g(G)\geq k$.
\end{theorem}

\begin{theorem}
For an arbitrary positive integer $k$ there exists a graph $G$ such that $\chi_{ig}(G)-\chi_g(G)\geq k$.
\end{theorem}

A natural question that arises is whether there are also graphs that have the game coloring number (respectively, the game chromatic number) arbitrarily larger than the independence game chromatic number. 
We will answer these questions in the affirmative for both mentioned invariants.

First, we mention that for the game coloring number the desired result can be obtained by using a theorem of Bartnicki et al.~\cite{bartnicki}, which states that for an arbitrary  positive integer $k$ there exists a positive integer $n$ such that $col_g(K_{1,n}\square K_{1,n}) > k$.
It is easy to prove that $\cigA(K_{1,n}\square K_{1,n})=2=\cigAB(K_{1,n}\square K_{1,n})$ for every $n$. In addition, with some more effort one can verify that $\cigB(K_{1,n}\square K_{1,n})\leq 4$ and $\cigBA(K_{1,n}\square K_{1,n})\leq 4$.
All these results imply that the difference $col_g(G)-\cig(G)$ can be made arbitrarily large for some graphs $G$ (note that $\cig$ stands for any of the four independence coloring game invariants). Nevertheless, we will prove the stronger result that also $\chi_g(G)-\cig(G)$ can be made arbitrarily large, which clearly implies the former result.

\begin{lemma}
\label{lem:chig-G3}
The game chromatic number of the graph $G_3^k$ is at least $k$.
\end{lemma}
\begin{proof}
Assume that Alice and Bob play a coloring game on $G_3^k$ with the set of colors $\{1,\ldots , m\}$, where $m<k$. Bob's goal is to create a situation in which there is an uncolored vertex with all colors from $\{1,2,\ldots ,m\}$ in its neighborhood, which he achieves with the following strategy. Regardless of the first move of Alice, there exists a set $Z_i$ in $Z$, such that no vertices from $Z_i$ and no vertices from $N(Z_i)$ have been colored in the first move. Let $Y'=N(Z_i)$. Note that $|Y'|=2k$. Bob's goal is to make one of the vertices in $Y'$ in the situation described above. For this purpose he colors a vertex of $Z_i$ with color $1$, by which no vertex from $Y'$ will receive color $1$. If Alice colors in one of her next moves the vertex $x$ with color $1$, she gains nothing, but if she colors it with one of the colors in $\{2,3,\ldots ,m\}$, she even helps Bob to accomplish his goal.  Further, if Alice plays vertices from $Z_i$ (say with color $1$), there are enough vertices in $Z_i$ (notably $|Z_i|=2k$), so that Bob can choose them and color them by different colors. Hence, the only possible strategy of Alice to prevent Bob from achieving his goal is to play in some of the vertices in $Y'$, and color them with a new color, say color $2$. By this she also prevents that this color (color $2$) is used by Bob in vertices of $Z_i$. However, Bob's strategy is then as follows. Let $Y''$, where $Y''\subset Y'$ be the set of vertices selected by Alice until a given point in the game. Since it is Bob's turn, he selects a set $Z_i$ such that $N(Z_i)\cap Y''=\emptyset$, and chooses the last color that was used by Alice in $Y''$ if this color is different from the colors that have been chosen by Bob in $Z$ (otherwise, Bob chooses any new color). Such a set $Z_i$ exists since after the $j$th move of Alice, at most $j$ vertices have been selected in $Y$, where $j\le m<k<|Y'|=2k$, and so there is a set $Z_i$, whose vertices are adjacent to all vertices of $Y'-Y''$ (and to some other vertices of $Y$ not colored by Alice). In this way, after Bob's $m$th move there is a vertex of $Y'$ in whose neighborhood all $m$ colors appear. Hence Bob wins the game on $G_3^k$, and $\chi_g(G_3^k)\ge k$.\qed \end{proof} 

By Lemma~\ref{lemma_G3}, we have $\cigA(G_3^k)=\cigAB(G_3^k)=2$, and by Lemma~ \ref{lem:G3-BA}, we have $\cigBA(G_3^k)\le 4$. Combining these observations with  Lemma~\ref{lem:chig-G3}, we infer the following result.

\begin{corollary}
There exist graphs $G$ such that $\chi_g(G)-\cigA(G)$, $\chi_g(G)-\cigAB(G)$, and  $\chi_g(G)-\cigBA(G)$ is arbitrarily large. 
\end{corollary}

Let $k$ be an arbitrary positive integer, and let $K$ be a copy of the complete bipartite graph $K_{k,k}$. A graph $G_4^k$ is obtained from $K$ by deleting the edges of a perfect matching $M$ in $K$, and adding two (non-adjacent) vertices $u$ and $v$ and make them adjacent to all vertices of $K$. Let the partite sets of $K$ be $A=\{a_1,\ldots ,a_{k}\}$ and $B=\{b_1,\ldots ,b_{k}\}$, the matching $M=\{a_1b_1,\ldots ,a_{k}b_{k}\}$, and $K'=K-M$.

\begin{lemma}
\label{lema_gB}
$\chi_g(G_4^k)=k+2$ and $\cigB(G_4^k)=3$.
\end{lemma}

\begin{proof}
First, let us prove that $\chi_g(G_4^k)=k+2$. Suppose that Alice and Bob are playing a coloring game on $G_4^k$ with the set of colors $[k+2]$. We prove that Alice wins the game. She starts the game by coloring $u$ with color $1$.
It is easy to see that it is not in Bob's favour to be the first one to play in $K'$.
Hence, Bob colors $v$ with color $2$ in his second move, since he wants to maximize the number of colors in the game. (Note that colors $1$ and $2$ can no longer be used in the game, since $u$ and $v$ are adjacent to all vertices of $K'$.)  Now, it is Alice's turn, and she colors a vertex $a_1$ in $K'$ with color $3$. Bob's optimal next move is to color the opposite vertex $b_1$ with $3$. (Indeed, if Bob chooses a vertex other than $b_1$, Alice colors $b_1$ with any color other than $3$, say color $q$, by which no vertex of $A$ will be able to get color $q$. This implies that all vertices of $A$ may receive color $q$ and all vertices of $B$ may receive color $3$, so Alice wins the game.)
In her next move, Alice colors $a_2$ with color $4$, and Bob follows by coloring the $b_2$ with color $4$ to prevent Alice from ensuring an easy win as described in the previous move. In this way, they continue, until all vertices are properly colored with colors $1,2,\ldots ,k+2$, where the vertices of $K'$ that are connected along a (missing) edge in $M$ get the same color.  Clearly, Alice cannot win this game on $G$ if less than $k+2$ colors are available. The strategy of Bob if only $m$, $m<k+2$, colors are available, is the same as described, yet, after $m$ pairs of vertices in $K'$ are given different colors, there are still $2k-2m$ vertices of $K'$ that have not been colored, and have all $m$ colors in their neighborhoods. Thus, $\chi_g(G_4^k)=k+2$.

Now, let us prove that $\cigB(G_4^k)=3$. Assume that Alice and Bob play a B-independence coloring game on $G_4^k$. Bob starts each round of the game. Without loss of generality, suppose that he selects as the first move of the game the vertex $u$. Hence, no other vertex except $v$ can be selected in round $1$, and $v$ must be selected by Alice in her first move. Therefore Bob starts the second round of the game, and he is forced to start playing in $K'$. Without loss of generality he selected $a_1$. Alice's response is to select another vertex of $A$, say $a_2$, after which all vertices of $A$ must be selected in round $2$.  Consequently, with round $3$ the game ends, since only vertices of (the independent set) $B$ remain uncolored. Thus, $\cigB(G_4^k)=3$. \qed
\end{proof}

\begin{corollary}
There exist graphs $G$ such that $\chi_g(G)-\cigB(G)$ is arbitrarily large. 
\end{corollary}

\section{Concluding remarks}

In this paper, we introduced a new two-player game on graphs, which is related to graph coloring. As with any newly introduced invariant, many open problems arise concerning the independence game chromatic numbers.

Due to the close relation of the first round of the independence coloring game to the independent domination game (alias, the competition-independence game), a natural question is when does these two games (the first round of our game and the entire independence domination game) coincide. The reason that this is not clear lies in the possibility that Alice may benefit from not obtaining a largest possible independent set in the first round, since this may help her in the next rounds of the independence coloring game. 

Several open problems are related to results of Section~\ref{sec:extremal}. Theorems~\ref{izrek_dva} and~\ref{izrek_dvaB} present characterizations of graphs with the independence game chromatic numbers equal to $2$, hence the following problem is natural.

\begin{problem}
For any invariant $\widetilde{\cig}\in\{\cigA,\cigB,\cigAB,\cigBA\}$, characterize the graphs $G$ for which $\widetilde{\cig}(G)=3$.
\end{problem}

In Section~\ref{sec:extremal} we also established that in split graphs $G$, $\cigA(G)=\omega(G)$, which triggers the following two problems. 

\begin{problem}
\label{prob}
For any invariant $\widetilde{\cig}\in\{\cigA,\cigB,\cigAB,\cigBA\}$, characterize the graphs $G$ for which $\widetilde{\cig}(G)=\omega(G)$.
\end{problem}

\begin{problem}
\label{probsplit}
For any invariant $\widetilde{\cig}\in\{\cigB,\cigAB,\cigBA\}$, characterize the split graphs $G$ for which $\widetilde{\cig}(G)=\omega(G)$.
\end{problem}

\noindent Problem~\ref{prob} could also be modified so that $\omega$ is replaced by $\chi$, yielding a superclass of the graphs with $\widetilde{\cig}(G)=\omega(G)$. Concerning Problem~\ref{probsplit}, note that $\widetilde{\cig}(G)$ either equals $\omega(G)$ or $\omega(G)+1$.

The following two questions were mentioned in Section~\ref{sec:compare}.

\begin{question}
Is there a graph $G$ such that $\cigAB(G)-\cigA(G)>0$, and if so, how large can the difference $\cigAB(G)-\cigA(G)$ be?
\end{question}

\begin{question}
Is there a graph $G$ such that $\cigBA(G)-\cigA(G)>0$, and if so, how large can the difference $\cigAB(G)-\cigA(G)$ be?
\end{question}

\noindent If the answers to the above questions are negative, it somehow confirms the intuition that starting each round is often not beneficial for Alice. 

In light of the results of Section~\ref{sec:compare}, which show that for almost all ordered pairs of the independence game chromatic numbers their difference can be arbitrarily large, it would be interesting to find if in some classes of graphs the invariants behave more tamely. For instance, are there some nice classes of graphs in which all four invariants have the same value?

%We end the paper with the following problem, which is in a sense related to the proof of the main theorem, Theorem~\ref{thm:trees}. 

%\begin{problem}
%How do the independence coloring games behave under the disjoint union of graphs?
%\end{problem}

\section*{Acknowledgements}

The authors acknowledge the financial support from the Slovenian Research Agency (research core funding No.\ P1-0297 and research projects J1-9109, J1-1693 and N1-0095).


\begin{thebibliography}{99}
\bibitem{bartnicki} T. Bartnicki, B. Bre\v{s}ar, J. Grytczuk, M. Kov\v{s}e, Z. Miechowicz, I. Peterin, Game chromatic number of Cartesian product graphs, Electron.\ J.\ Combin.\ 15 (2008) \#R72, 13pp.

\bibitem{bagr-07} T.~Bartnicki, J.~Grytczuk, H.~A.~Kierstead, X.~Zhu, The
map coloring game, Amer.\ Math.\ Monthly 144 (2007) 793--803.

\bibitem{bo-1991} H.L.~Bodlaender, On the complexity of some coloring games, Internat.\ J.\ Found.\ Comput.\ Sci.\ 2 (1991) 133--147.

\bibitem{bosek} B.~Bosek, J.~Grytczuk, G.~Jak\'{o}bczak, 
Majority coloring game, Discrete Appl.\ Math.\ 255 (2019) 15--20.

\bibitem{brklra-2010}
  B.~Bre{\v{s}}ar, S.~Klav{\v{z}}ar, D.~F.~Rall,
  Domination game and an imagination strategy,
  SIAM J. Discrete Math. 24 (2010) 979--991.


%\bibitem{dizh-99} T.~Dinski, X.~Zhu, Game chromatic number of graphs,
%Discrete Math.\ 196 (1999) 109--115.

\bibitem{faigle} U. Faigle, U. Kern, H.~A. Kierstead, W.~T.~Trotter, On the game chromatic number of some classes of graphs, Ars Combin.\ 35 (1993) 143--150.

\bibitem{ga-81}
  M.~Gardner,  Mathematical games,  Scientific American 244 (1981) 18--26.


\bibitem{go-he-2018}
W. Goddard, M.~A. Henning, The competition-independence game in trees, J.\ Combin.\ Math.\ Combin.\ Comput.\ 104 (2018) 161--170.

\bibitem{gr-2012} A.~Grzesik, Indicated coloring of graphs, 
 Discrete Math.\ 312 (2012) 3467--3472.
 
\bibitem{hammer-81} P.~L. Hammer, B. Simeone, The splittance of a graph, Combinatorica 1 (1981) 275-284.

\bibitem{kiko-2009} H.A.~Kierstead, A.~Kostochka, Efficient graph packing via game coloring, Combin.\ Probab.\ Comput.\ 18 (2009) 765--774.

\bibitem{kir-2012}
 H.A.~Kierstead, C.-Y.~Yang, D.~Yang, X.~Zhu,
Adapted game colouring of graphs, European J.\ Combin.\ 33 (2012) 435--445.

\bibitem{KWZ-2013}
  W.~B.~Kinnersley, D.~B.~West, R.~Zamani,
  Extremal problems for game domination number,
  SIAM J.\ Discrete Math.\ 27 (2013) 2090--2107.
  
\bibitem{klra-2019}
  S.~Klav{\v{z}}ar, D.~F.~Rall,
  Domination game and minimal edge cuts,
  Discrete Math.\ 342 (2019) 951--958.


\bibitem{mpw-2018} T. Mahoney, G.J.~Puleo, D.B.~West, 
Online sum-paintability: the slow-coloring game, 
Discrete Math.\ 341 (2018) 1084--1093.

%\bibitem{} R.~M.~R. Lewis, A Guide to Graph Colouring: Algorithms and Applications, Springer International Publishing (2016).

\bibitem{ph-sl-2001}
J.~B. Phillips, P.~J. Slater, An introduction to graph competition independence and enclaveless parameters, Graph Theory Notes N.\ Y.\ 41 (2001) 37--41.


\bibitem{tu-2016} Zs. Tuza, X. Zhu, Colouring games, in [Topics in chromatic graph theory, Cambridge Univ. Press, Cambridge, 2016], 304--326.



\end{thebibliography}
\end{document}